\documentclass[12pt]{amsart}

\usepackage[english]{babel}
\usepackage[utf8]{inputenc}
\usepackage{amsmath}
\usepackage{amssymb}
\usepackage{amsfonts}
\usepackage{amsthm}
\usepackage{mathrsfs}
\usepackage[all]{xy}
\usepackage[pdftex]{graphicx}
\usepackage{color}
\usepackage{cite}
\usepackage{url}
\usepackage{indent first}
\usepackage[labelfont=bf,labelsep=period,justification=raggedright]{caption}
\usepackage[english]{babel}
\usepackage[utf8]{inputenc}
\usepackage{hyperref}
\usepackage[colorinlistoftodos]{todonotes}
\usepackage{tkz-fct}
\usepackage{tikz}

\DeclareMathOperator{\im}{Im}

\DeclareMathOperator{\re}{Re}
\DeclareMathOperator{\SL}{SL}

\newcommand{\CC}{\mathbb{C}}

\newcommand{\FF}{\mathbb{F}}

\newcommand{\NN}{\mathbb{N}}
\newcommand{\PP}{\mathbb{P}}
\newcommand{\QQ}{\mathbb{Q}}
\newcommand{\RR}{\mathbb{R}}
\newcommand{\ZZ}{\mathbb{Z}}

\newcommand{\mcC}{\mathcal{C}}
\newcommand{\mcF}{\mathcal{F}}

\newcommand{\mcH}{\mathcal{H}}

\theoremstyle{plain}
\newtheorem{thm}{Theorem}
\newtheorem{lemma}[thm]{Lemma}
\newtheorem{cor}[thm]{Corollary}

\newtheorem{prop}[thm]{Proposition}

\theoremstyle{definition}
\newtheorem{defn}[thm]{Definition}

\theoremstyle{remark}
\newtheorem*{rem}{Remark}

\numberwithin{equation}{section}
\numberwithin{thm}{section}

\begin{document}

\title{Infinite Product Exponents for Modular Forms}

\author{Asra Ali}
\address{Massachusetts Institute of Technology, Department of Mathematics, Cambridge, MA}
\email{asra@mit.edu}

\author{Nitya Mani}
\address{Stanford University,  Department of Mathematics, Stanford, CA 94305}
\email{nityam@stanford.edu}
\thanks{The authors would like to thank the NSF and the Emory REU (especially Dr. Mertens and Professor Ono) for their support of our research.}
\date{\today}



\maketitle

\section{Introduction} \label{Introduction}

The study of modular forms is a rich and deep subject with connections to mathematics ranging from partitions to elliptic curves. In particular, examinations of modular forms are often informed by their divisors, most easily determined from the infinite product expansion of the modular form.
Any holomorphic modular form $f(z) = \sum_{n = h}^{\infty} a(n)q^n$ where $q = e^{2\pi i z}$ and $a(h) = 1$ can be written as an infinite product $$f(z) = q^h \prod_{m = 1}^\infty (1 - q^m)^{c(m)}; \qquad c(m) \in \CC.$$

For example, consider the infinite product expansions of a family of such modular forms intimately connected to a modular form called the Dedekind $\eta$-function. The Dedekind $\eta$-function is a weight $1/2$ modular form defined by
\[\eta(z) = q^{1/24} \prod_{m=1}^\infty (1-q^m) .\] 

One way to construct cusp forms of level $N$ is through eta-quotients, expressions of the form
\[ f(z) = \prod_{d|N} \eta(dz)^{r_d} \]
where $r_d$ is some integer depending on the divisor $d$. Some holomorphic modular forms can be expressed as an eta-quotient, giving the values $c(m)$ in their infinite product expansion. However, only finitely many weight $2$ newforms can be as an eta-quotient, completely characterized by Ono and Martin (see \cite{MAR97}).  Clearly, in this case, the values $c(m)$ are bounded.

However, $c(m)$ is not a repeating sequence (or even bounded) for generic modular forms $f(z)$. Understanding the function $c(m)$ can help determine the divisors of the associated modular forms. These also appear in the theory of Borcherds products. Bruinier, Kohnen, and Ono first derived an expression for $c(m)$ as a function of $m$ for an associated meromorphic modular form defined over $\SL_2(\ZZ)$ (see \cite{ONO04}). Later work by Ahlgren and Choi generalized this result to meromorphic modular forms defined over $\Gamma_0(N)$ (see \cite{AHL03,CHOI10}). Related work by Movasati and Nikdelan also motivated this study. Further, Kohnen (see \cite{KOH05}) established a bound on the growth of $c(m)$ independent of these formulations if $f(z)$ has no zeros or poles on the upper half plane. He showed that given this condition, if $f(z)$ is a modular form over any finite index subgroup of $\SL_2(\ZZ)$, $c(m) \ll_f \log \log n \cdot \log n$ and that if $f(z)$ is a modular form over any congruence subgroup of $\SL_2(\ZZ)$, $c(m) \ll_f (\log \log n)^2$. Here we give tight bounds on $c(m)$ for some infinite classes of modular forms and an upper bound on the growth of $c(m)$ for any holomorphic modular form $f(z)$.

Throughout the paper, let $f(z) = \sum_{n = h}^{\infty} a(n)q^n$, $a(h) = 1$, be a holomorphic modular form of weight $k \in \ZZ_{\ge 0}$ over $\Gamma_0(N)$. Let $\mcF_N$ be a fundamental domain for the action of $\Gamma_0(N)$ on $\mcH$. The infinite product expansion for $f(z)$ is written as 

\begin{align}
f(z) = q^h \prod_{m = 1}^{\infty} (1 - q^m)^{c_f(m)}; \qquad c_f(m) \in \CC. \label{e:feqn}
\end{align}

\begin{thm}\label{c:upperBound}
Assume the set of roots of $f(z)$ (with infinite product expansion as in \eqref{e:feqn}) in a fundamental domain $\mcF_N$ is $\{z_j = x_j + iy_j \}_{j = 1, ..., r}$ with $y_1 \le ... \le y_r$ and $r \ge 1$. Then we have that $$c_f(m) \ll \frac{e^{2\pi m y_r}}{m^{3/2}}.$$
\end{thm}

If the genus of $X_0(N)$ is $0$ or $1$, then we also obtain a lower bound on $c(m)$. In particular, we obtain an $\Omega$ bound, defined as follows. Given two functions $f, g$, $f = \Omega(g)$ implies that there exists some positive constant $c$ such that for all $n_0 \in \NN$ there exists infinitely many $n > n_0$ so that $f(n) \ge c\cdot g(n)$.

\begin{thm}\label{t:mainThm}
Assume the set of roots of $f(z)$ in $\mcF_N$ is $\{z_j = x_j + iy_j \}_{j = 1, ..., r}$ with $y_1 \le ... \le y_r$ and $r \ge 1$.
If $f(z)$ is a modular form for $\Gamma_0(N)$ such that the genus of $X_0(N)$ is $0$ or $1$, then we have that $$c_f(m) = \Omega\left(\frac{e^{2\pi m y_r}}{m^{3/2}}\right).$$
\end{thm}

\begin{rem}
It is natural to ask what the dependence of this result on the genus of $X_0(N)$ is. The above result follows when the cusp forms of weight $2$ and level $N$ have infinitely many vanishing Hecke eigenvalues.
\end{rem}

From the above results, we can also obtain a similar (slightly weaker) bound to that in \cite{KOH05} stated above.


\begin{cor}\label{c:kohnen}
Suppose that $f(z)$ is a modular form with no zeros or poles on the upper half plane. Then, we obtain $c_f(m) \ll \log m \cdot \log \log m$.
\end{cor}

\section{Preliminaries}
Let $q = e^{2 \pi i z}$, with $z = x + iy \in \mcH$. Consider a modular form $f(z)$ with Fourier expansion $f(z) = \sum_{n = h}^{\infty} a(n) q^n$ with $a(h) = 1$ for a congruence subgroup
 $$\Gamma_0(N) = \left\{ 
\left( \begin{matrix}
a & b \\
c & d \\
\end{matrix} \right) \in \SL_2(\ZZ) \,\,|\,\, c \equiv 0 \mod N
\right \}.$$
Under this group action, two elements $z_1, z_2 \in \mcH \cup \PP_1(\QQ)$ are equivalent, denoted $z_1 \sim z_2$, when there exists a $\sigma \in \Gamma_0(N)$ such that $\sigma z_1 = z_2$. Denote a set of representatives of the inequivalent cusps of $\Gamma_0(N)$ by $\mcC_N$, with $\mcC_N^{*} = \mcC_N \backslash \{\infty \}$. Then consider the modular curve of level $N$,  \[X_0(N) = \Gamma_0(N) \, \backslash \, ( \mcH \cup \PP^1(\QQ)).\] Let $\nu_{z}^{(N)}(f(z))$ be the (weighted) order of the zero of $f(z)$ at $z$ on $X_0(N)$.
\begin{prop}
Let $f(z)$ be a modular form for $\Gamma_0(N)$. Define $$f_{\theta}(z) := \frac{\theta f(z)}{f(z)} + \frac{k/12 - h}{N - 1} \cdot N E_2(Nz) + \frac{h - N k/12}{N-1} \cdot E_2(z)$$ where $$\theta f(z) = \sum_{n = h}^{\infty} n a_n q^n,$$ is the Ramanujan $\theta$ operator, $$E_2(z) = 1 - 24\sum_{n \ge 1} \sigma_1(n) q^n$$ is the normalized quasimodular Eisenstein series of weight $k = 2$, and $\sigma_j(n) = \sum_{d | n} d^j$ is the sum of the $j$th powers of divisors of $n$. Then $f_\theta(z)$ is a meromorphic modular form of weight $2$ for $\Gamma_0(N)$. 
\end{prop}

Consider the weight $0$ index $m$ Poincaré series (Theorem 1 in \cite{NIE73}) for all $z \in \mcH$ and $s \in \CC$ with $\re(s) > 1$: \begin{align}\label{e:poincare}F_{N, m}(z, s) = \sum_{\gamma \in \Gamma_0(N)_{\infty} \backslash \Gamma_0(N) } \pi \sqrt{|\im(\gamma z)|} I_{s- \frac12} (|2\pi m \im(\gamma z)|)e^{- 2\pi i m \re(\gamma z)},\end{align} where 
$I_{\nu}(x)$ is the usual modified $I$-Bessel function of order $\nu$. 
\begin{prop}[\S1 in \cite{CHOI10}]
Define $j_{N, m}(z)$ to be the analytic continuation of $F_{N, m}(z, s)$ (where $\re(s) \le 1$) as $s \rightarrow 1^+$. Then, $j_{N, m}(z)$ is the constant term of the Fourier expansion of $F_{N, m}(z, 1)$ when $t \in \mcC_N^{*}$.
\end{prop} 
\begin{prop}[p. 30 \cite{ONO09}] Define a differential operator $\xi_0$ on the space of functions $j_{N, m}$, $$\xi_0(j_{N, m}) := 2i \overline{\frac{\partial}{\partial \overline{z}}j_{N, m}(z) }.$$ The differential operator $\xi_0$ maps $j_{N,d}$ to a cusp form.
\end{prop}
Here, we consider the regularized integral (since $f_{\theta}(z)$ is generally not holomorphic on $\mcH$) $$\int_{\mcF_N}^{reg} f_{\theta}(z) \cdot \xi_0(j_{N, m}(z)) dx dy$$ where the regularization method can be found in \cite{CHOI10}.

\section{Proof of Main Result}

Given this setup, the exponents of the infinite product expansion of a modular form can be obtained by applying the Möbius inversion formula to the result obtained in \cite{CHOI10}.

\begin{prop}\label{p:c(m)}
Consider a normalized (holomorphic) modular form of weight $k$ on $\Gamma_0(N)$ for $N > 1$, $f(z) = q^h + \sum_{m = h+1}^{\infty} a(m)q^m$ and denote by $\{c_f(m)\}_{m = 1}^{\infty}$, $c_f(m) \in \CC$ the exponents such that $$f(z) = q^h \prod_{m = 1}^{\infty} (1 - q^m)^{c_f(m)}.$$ Then, we have that  
$$c_f(m) = \frac{1}{m} \sum_{d |m} \mu \left( \frac{m}{d} \right) \left( \sum_{z \in \mcF_N \cup \mcC_N^{*}} \nu_{z}^{(N)} (f) j_{N, d}(z) - \int_{\mcF_N}^{reg} f_{\theta}(z) \cdot \xi_0(j_{N, d}(z)) dx dy \right) $$ $$ +   \begin{cases} 
      \, \frac{2Nk - 24h}{N - 1}    \hfill & N \hspace{-6pt} \not | m \\
      \quad 2k \hfill & N \hspace{1pt} | m \\
  \end{cases}$$ where $\mu$ is the Möbius function.
\end{prop}

A normalized modular form with integral coefficients $f(q) = \sum_{n = h+1}^{\infty} a(n)q^n$ has a product expansion $f(z) = q \prod_{m = 1}^\infty (1-q^m)^{c_f(m)}$ where the $c_f(m)$ are integers. This can be seen by expanding the infinite product and solving for the coefficients of the Fourier expansion. Now we consider $c_f(m)$ as $m \rightarrow \infty$ by computing the growth of terms in the above expression.

\subsection{Growth of $j_{N, m}(z)$}
We note that $F_{N,m}(z,s)$ is a harmonic Maa\ss\hspace{2pt} form (defined below) and cite a Lemma giving an explicit computation for the values for the analytic continuation $j_{N,m}(z)$.

\begin{defn}[Definition 7.1 \cite{ONO09}]
A smooth function $f: \mcH \rightarrow \CC$ is a \textit{weak harmonic Maa\ss\hspace{2pt} form} of weight $k$ for $\Gamma_0(N)$ if the following conditions are satisfied: \begin{enumerate} \item $f$ transforms like a modular form under the action of $\Gamma_0(N)$, \item $f$ is in the kernel of the weight $k$ hyperbolic Laplacian $\Delta_k := -y^2 \left(\frac{\partial^2}{\partial x^2} + \frac{\partial^2}{\partial y^2} \right) + iky\left(\frac{\partial}{\partial x} +i\frac{\partial}{\partial y} \right)$, \item $f$ has at most linear exponential growth at the cusps of $f$. \end{enumerate}
\end{defn}

\begin{lemma}[Theorem 2.1 \cite{CHOI10}]\label{l:hwm}
The function $F_{N, m}(z, 1)$ is a weak harmonic Maa\ss\hspace{2pt} form of weight $0$ on $\Gamma_0(N)$. Further at $F_{N, m}(z, 1)$ has the following properties at each cusp $t$: If $t \sim \infty$, 
$$F_{N, m}(z, 1) = q^{-m} + \sum_{n \ge 0} b_m(n, 1)q^n + \sum_{n > 0} b_m(-n, 1)e^{2\pi i n\overline{z}}$$ and otherwise if $t \not \sim \infty$, 
$$\lim_{\im z \rightarrow \infty} F_{N, m}(\sigma_tz, 1) = j_{N, m}(t)$$ where $\sigma_t \in \SL_2(\RR)$ is defined so that $\sigma_t^{-1}\infty = t$ and $\sigma_t \Gamma_0(N)_t \sigma_t^{-1} = \left\{ \pm \begin{bmatrix}
1 & n \\
0 & 1 \\
\end{bmatrix} | n \in \ZZ
\right \}$
\end{lemma}

\begin{lemma}\label{l:jgrowth}
Let $z = x + iy \in \mcF_N \cup \mcC_N^*$. As $m \rightarrow \infty$, we have that $j_{N, m}(z) \asymp \left(\frac{e^{2\pi m y}}{\sqrt{m}} \right)$
\end{lemma}
\begin{proof}
From \cite{NIE73}, we can obtain the analytic continuation of $F_{N, m}(z, s)$ where $\re(s) > 1/2$ and a Fourier expansion for this function as $s \rightarrow 1^+$ which we term $F_{N,m}(z, 1)$. Since $z \not \sim \infty$, $j_{N, m}(z)$ is the constant term of this Fourier expansion of $F_{N,m}(z, 1)$ and can be computed as $$j_{N, m}(z) = e^{2\pi i |m| x} \sqrt{y} I_{1/2} (2\pi |m| y) + a_{m}(1)\sqrt{y} K_{1/2} (2\pi |m| y)$$ 
where $a_m(1)$ is the $m$th coefficient of the Maa\ss-Eisenstein series for $\Gamma_0(N)$ and $K_{\nu}(z), I_{\nu}(z)$ are the modified Bessel functions of the first and second kind with special values given above. Then, apply the relations on these modified Bessel functions, noting that (see eq. 10.49 in \cite{OLV10}), $$I_{1/2}(z) = \frac{\sinh |z|}{\sqrt{\frac{\pi |z| }{2}}}; \qquad K_{1/2}(z) = \frac{e^{-|z|}}{\sqrt{\frac{\pi |z|}{2}}}$$
Using these expressions we evaluate:
$$j_m(z) = \frac{1}{\pi \sqrt{m}} \left( e^{2\pi i m x} \sinh(2 \pi m y) + a_n(1) e^{ - 2\pi m y} \right).$$
If $z \in \mcC_N^*$, $y = 0$, and thus $j_m(z) \asymp \frac{1}{\sqrt{m}} \ll 1$. Else if $z \in \mcH$, $y > 0$. Taking $m \rightarrow \infty$, 
and simplifying the above expression, we retrieve the desired result: \begin{align*}
\lim_{m \rightarrow \infty} j_m(z) 
&\asymp \frac{e^{2\pi m y}}{\sqrt{m}}.
\end{align*}
\end{proof}


\subsection{Regularized Integration}
We examine the growth of the regularized integral as in Proposition~\ref{p:c(m)}:
\begin{equation}\label{e:reg} R(m) = \int_{\mcF_N}^{reg} f_{\theta}(z) \cdot \xi_0(j_{N, m}(z)) dx dy. \end{equation}

First, we compute some cases where this term vanishes, so that the contribution from this term in the growth of $c_f(m)$ is $0$. The following lemma describes this.

\begin{lemma}\label{l:reg}
\begin{enumerate}
\item If the genus of $X_0(N)$ is 0, the regularized integral as defined above vanishes. 
\item If genus of $\Gamma_0(N)$ is 1, the regularized integral vanishes for infinitely many positive integers $m$.
\end{enumerate}
\end{lemma}
\begin{proof}
Recall that $\xi_0(j_{N,d}(z))$ is a weight two cusp form. Suppose that the genus of $X_0(N)$ is zero. Then the space of cusp forms of weight two is trivial, so $\xi_0(j_{N, d}(z))$ must be $0$. It follows that the regularized integral vanishes.

Now let the genus of $X_0(N)$ be one. The following argument appears as a remark in \cite{CHOI10}. The space of weight two cusp forms of weight two is spanned by a unique normalized weight $2$ cusp form $g(z) = \sum_{n = 1}a(n)q^n$. There exists an elliptic curve $E_g$ of conductor dividing $N$ whose Hasse-Weil $L$-series coincides with $L$-function for $g(z)$. In other words, for all values of $p$ that do not divide $N$,
\[ p+1-a(p) = \#E_g/\FF_p. \]
If $E_g$ is supersingular at $p$, the number of points of $E_g/ \FF(p)$ is exactly $p+1$ and $a(p) = 0$. So we have that for an odd $m$, $j_{N, m}(z)$ identically vanishes if and only if $E_g$ is supersingular at some $p|m$. Since there exist infinitely many supersingular primes for every elliptic curve over $\QQ$ (see \cite{ELK87}), there are infinitely many positive integers $m$ such that $j_{N, m}(z)$ is holomorphic on $\mcH$ (i.e. $\xi_0 j \equiv 0$).
\end{proof}

This regularized integral $R(m)$ may not vanish if the genus of $X_0(N)$ is not $0$. Below we address the growth of $R(m)$. First, we characterize the growth of the Fourier cofficient $b_m(1,0)$ of the weight $0$ Poincaré series of index $m$.

\begin{lemma}\label{l:bcoeff}
Recall the weight $0$ and index $m$ Poincaré series given in Equation~\ref{e:poincare} with Fourier expansion
\[ F_{N,m}(z,s) = q^{-m} + \sum_{n \geq 0}b_m(n,1)q^n + \sum_{n > 0} b_m(-n, 1)e^{2\pi i n \overline{z}} .\]
Then $b_m(1,0) \ll  m^{1/4}e^{4\pi\sqrt{m}}$.
\end{lemma}
\begin{proof}
Theorem 8.4 in \cite{ONO09} gives the following expansion of $b_m(1,0)$:
\[b_m(1,0) = 2\pi\left(m\right)^{1/2} \sum_{c>0, N|c} \frac{K_0(-m, 1, c)}{c}I_1\left( \frac{4\pi \sqrt{|m|}}{c}\right).\]
where $K_0(-m, n, c)$ is the Kloosterman sum defined as $$K(a, b; m) = \sum_{0 \le d \le m - 1, \gcd(d, m) = 1} e^{\frac{2\pi i}{m}(ad + bd^*)}$$ where $d^*$ is the multiplicative inverse of $d$ modulo $m$. Weil's bound gives 
\[b_m(1,0) \ll m^{1/2} \sum_{c>0, N|c} \frac{c^{1/2 + \epsilon}}{c}I_1\left( \frac{4\pi \sqrt{|m|}}{c}\right). \]
Now we expand the modified Bessel function $I_1\left( \frac{2\pi \sqrt{|m|}}{c} \right)$.
\[I_1\left( \frac{4\pi\sqrt{|m|}}{c} \right) = \frac{4\pi \sqrt{|m|}}{c} \sum_{k = 0}^\infty \frac{(\pi^2 |m|/c^2)^k}{k!(1+k)!}.\]
We can now examine the growth $I_1$ of $b_m(1,0)$ as $m \rightarrow \infty$. The Bessel function is approximated by
\[I_1\left( \frac{4\pi\sqrt{|m|}}{c} \right) \sim \frac{e^{(4\pi\sqrt{|m|})/c}}{\sqrt{8 \pi^2 \sqrt{|m|}/c}} \] for sufficiently large $m$ (see 10.41 in \cite{OLV10}).
Then as $m \rightarrow \infty$, we obtain
\[b_m(1,0) \ll m^{1/4}e^{4\pi\sqrt{m}}. \]
\end{proof}

\begin{lemma}\label{l:reggrowth}
Consider the regularized integral in~\ref{e:reg}. Suppose that the set of zeros of $f(z)$ in a chosen fundamental domain is $\{z_j = x_j + iy_j\}_{j = 1, \dots, r}$ with $y_1 \leq \dots \leq y_r$, $r \geq 1$. As $m \rightarrow \infty$ we can bound $R(m)$ as: \[R(m) \ll \frac{e^{2\pi m y_r}}{\sqrt{m}}.\]
\end{lemma}

\begin{proof}
The function $f_\theta(z)$ is a meromorphic modular form of weight two on $\Gamma_0(N)$, which is holomorphic at each cusp and each pole is simple. Thus, it satisfies the conditions for Lemma 3.1 in \cite{CHOI10}, giving the following expansion for $R(m)$.
\begin{align}
R(m) &= \lim_{\epsilon \rightarrow 0} \int_{\mcF_N(f_\theta, \epsilon)} f_\theta(z) \cdot \xi_0(j_{N, m}(z))dxdy \\
&= b_m(1,0)a(0) + a(m) + \sum_{t \in \mcC_N^*} \alpha_t f_\theta(t) j_{N,m}(t)
+ \sum_{t \in S(f_\theta)} \frac{2\pi i}{\ell_t} \text{Res}_t(f_\theta) j_{N,m}(t) \label{e:rexp}
\end{align}
where $\alpha_t$ and $\ell_t$ are constants depending on $t$ defined in \cite{CHOI10}.
Lemma~\ref{l:bcoeff} shows that term $b_m(1,0)a(0) \ll m^{1/4}e^{4\pi\sqrt{m}}$. Moreover, the term $a(m)$ is negligible, since the coefficients of the modular form $f(z)$ has polynomial growth in $m$.  
Now, the cusps in $\mcC_N^*$ are rational points. From the proof of Lemma~\ref{l:jgrowth}, we have 
\begin{align*}
\lim_{m \rightarrow \infty} j_m(z) &\asymp \frac{e^{2\pi m y_r}}{\sqrt{m}}.
\end{align*}
For each $t \in \mcC_N^*$, $\im(t) = 0$, so this gives 
\begin{align}\label{term3}\sum_{t \in \mcC_N^*} \alpha_t f_\theta(t) j_{N,m}(t) \asymp \frac{1}{\sqrt{m}} . \end{align}
Now we examine the last term in this expression, a sum over the poles of $f_\theta(z)$. These are the zeros of $f(z)$. Since the $m$-dependence of this term also comes from $j_{N,m}(t)$, the growth of this term is dominated by the zeros of $f(z)$. Thus we have
\begin{align} \sum_{t \in S(f_\theta)} \frac{2\pi i}{\ell_t} \text{Res}_t(f_\theta)j_{N,m}(t) \asymp \frac{e^{2\pi i m y_r}}{\sqrt{m}} . \end{align}
as in the proof of Lemma~\ref{l:jgrowth}. This dominates the term in ~\eqref{term3}, so the growth of $R(m)$ is dictated by the above. Thus, given the integral expansion in ~\eqref{e:rexp}, $$R(m) \ll \frac{e^{2\pi m y_r}}{\sqrt{m}}.$$
\end{proof}



\subsection{Proof of Results}
We begin by proving the second theorem since many of the arguments used to prove the first theorem parallel the ones employed here:

\begin{proof}[Proof of Theorem~\ref{t:mainThm}]
From Proposition~\ref{p:c(m)}, we obtain $$c_f(m) = \frac{1}{m} \sum_{d |m} \mu \left( \frac{m}{d} \right) \left( \sum_{z \in \mcF_N \cup \mcC_N^{*}} \nu_{z}^{(N)} (f(z)) j_{N, d}(z) - \int_{\mcF_N}^{reg} f_{\theta}(z) \cdot \xi_0(j_{N, d}(z)) dx dy \right) $$ $$ +   \begin{cases} 
      \, \frac{2Nk - 24h}{N - 1}    \hfill & N \hspace{-6pt} \not | m \\
      \quad 2k \hfill & N \hspace{1pt} | m. \\
  \end{cases}$$
According to Lemma~\ref{l:reg}, since the genus of $X_0(N)$ is $0$ or $1$, the regularized integral defined in Section 3.2 vanishes infinitely often. Thus, for infinitely many $m$, $$c_f(m) = \Omega \left(  \frac{1}{m} \sum_{d |m} \mu \left( \frac{m}{d} \right) \left( \sum_{z \in \mcF_N \cup \mcC_N^{*}} \nu_{z}^{(N)} (f(z)) j_{N, d}(z) \right) +   \begin{cases} 
      \, \frac{2Nk - 24h}{N - 1}    \hfill & N \hspace{-5pt}\not | m \\
      \quad 2k \hfill & N | m. \\
  \end{cases} \right)$$ 
  
Note that the term depending on whether $N | m$ or not does not grow with $m$ and is bounded with respect to $m$. Then, since $m$ dominates the other divisors of $m$ as $m \rightarrow \infty$, we evaluate the expression $\sum_{z \in \mcF_N \cup \mcC_N^{*}} \nu_{z}^{(N)} (f(z)) j_{N, d}(z)$ only at $d = m$. Further $ \nu_{z}^{(N)} $ is a constant that does not depend on $m$. Thus we arrive at
$$c_f(m) = \Omega \left( \frac{1}{m}\sum_{z \in \mcF_N \cup \mcC_N^*} j_{N, m}(z) \right).$$ 

By Lemma~\ref{l:jgrowth}, we have $j_{N,m}(z) \asymp \frac{e^{2\pi m y}}{\sqrt{m}}$ for $z = x+iy$ in $\mcF_N \cup \mcC_N^*$. Then, the dominating term of the sum is the term with $z = z_r$, giving
$$c_f(m) = \Omega \left( \dfrac{e^{2\pi m y_r}}{m^{3/2}} \right). $$ 
\end{proof}

Now, suppose we relax the conditions on the genus of $X_0(N)$. We now prove the upper bound on $c_f(m)$ in general given in Theorem~\ref{c:upperBound}.

\begin{proof}[Proof of Theorem~\ref{c:upperBound}]
We examine the worst case growth of the equation
$$c_f(m) = \frac{1}{m} \sum_{d |m} \mu \left( \frac{m}{d} \right) \left( \sum_{z \in \mcF_N \cup \mcC_N^{*}} \nu_{z}^{(N)} (f(z)) j_{N, d}(z) - \int_{\mcF_N}^{reg} f_{\theta}(z) \cdot \xi_0(j_{N, d}(z)) dx dy \right) $$ $$ +   \begin{cases} 
      \, \frac{2Nk - 24h}{N - 1}    \hfill & N \hspace{-6pt} \not |\hspace{1pt} m \\
      \quad 2k \hfill & N \hspace{1pt}|\hspace{1pt} m. \\
  \end{cases}$$
  
  Like above, the term depending on whether $N|m$ or not is bounded as $m \rightarrow \infty$. Evaluating the remaining expression at $d = m$ we have
  \[ c_f(m) \ll \frac{1}{m}\left(\sum_{z \in \mcF_N \cup \mcC_N^{*}} \nu_{z}^{(N)} (f(z)) j_{N, m}(z) - \int_{\mcF_N}^{reg} f_{\theta}(z) \cdot \xi_0(j_{N, m}(z)) dx dy\right) + O(1).\]
  In the worst case, the regularized integral is $O(\frac{e^{2\pi m y_r}}{\sqrt{m}})$, by Lemma~\ref{l:reggrowth}. Combining this with Lemma~\ref{l:jgrowth}, we obtain
  \[ c_f(m) \ll \frac{1}{m}\left(\frac{e^{2\pi m y_r}}{\sqrt{m}} \right).\]
\end{proof}

Finally, we see that the bounds yielded by our calculation are consistent with earlier work in \cite{KOH05}.

\begin{proof}[Proof of Corollary~\ref{c:kohnen}]
Note that as $m \rightarrow \infty$, we have the bound $$c_f(m) \asymp \frac{1}{m} \sum_{d |m} \left( \sum_{z \in \mcF_N \cup \mcC_N^{*}} \nu_{z}^{(N)} (f(z)) j_{N, d}(z) - \int_{\mcF_N}^{reg} f_{\theta}(z) \cdot \xi_0(j_{N, d}(z)) dx dy \right) + O(1)$$
Since $f(z)$ has no roots or poles on the upper half plane, the first term in the above summation, $$\sum_{z \in \mcF_N \cup \mcC_N^{*}} \nu_{z}^{(N)} (f(z)) j_{N, d}(z)$$ reduces to $$\sum_{z \in \mcC_N^{*}} \nu_{z}^{(N)} (f(z)) j_{N, d}(z).$$
Here we exclude the cusp at $\infty$. By Lemma~\ref{l:jgrowth}, the above expression grows asymptotically proportional to $e^{2\pi m y_r}/\sqrt{m}$, where $y_r$ is the largest imaginary part of a summand with nonzero order. However, all of the cusps $z \not\sim \infty$ are on the rational line and thus have imaginary part $0$. Consequently, $$\sum_{z \in \mcC_N^{*}} \nu_{z}^{(N)} (f(z)) j_{N, d}(z) \asymp_{m \rightarrow \infty} \frac{1}{\sqrt{m}} \rightarrow 0.$$
Now consider the regularized integral. As per Theorem 1.4 in \cite{CHOI10}, $$f_{\theta}(z)_d + \int_{\mcF_N}^{reg} f_{\theta}(z)  \cdot \xi_0(j_{N, d}(z)) dx dy = \sum_{z \in \mcC_N^{*}} \nu_{z}^{(N)} (f(z)) j_{N, d}(z) $$ where $f_{\theta}(z)_d$ is the coefficient of $q^d$ in the Fourier expansion of $f_{\theta}(z)$. Using the above bound on the cusp summation, $$\int_{\mcF_N}^{reg} f_{\theta}(z) \cdot \xi_0(j_{N, d}(z)) dx dy \ll f_{\theta}(z)_d.$$
Now recall the definition $$f_{\theta}(z) = \frac{\theta f}{f} + \frac{k/12 - h}{N - 1} \cdot N E_2(Nz) + \frac{h - Nk/12}{N- 1}E_2(z)$$
Since $\frac{\theta f}{f}$ is holomorphic and quasi-modular of weight $2$, the coefficients of the Fourier expansion of $\frac{\theta f}{f}$ can be bounded by $(\frac{\theta f}{f})_d \ll d \log d$, (see \cite{MAA83}). Given the definition of the Eisenstein series, we can also bound the coefficients of $E_2(z)$ (and $\sigma_1(d)$ as in \cite{HAR79}): $$E_2(z)_d \asymp \sigma_1(d) \ll n \log \log n.$$
Combining these two bounds gives $$\int_{\mcF_N}^{reg} f_{\theta}(z) \cdot \xi_0(j_{N, d}(z)) dx dy \ll f_{\theta}(z)_d \ll d \log d + d \log \log d \ll d \log d.$$
Then, we can once again apply $\sigma_1(n) \ll n \log \log n$ to the summation over $d|m$ to compute the desired bound for $c_f(m)$: $$c_f(m) \ll \frac{1}{m} ((m \log \log m) \cdot \log m) + O(1) \ll \log m \cdot \log \log m .$$

\end{proof}


\end{document}